\documentclass[10pt,a4paper]{amsart}
          
\usepackage{mystyle}

\title{Murphy's law for toric vector bundles on smooth projective toric varieties}
\author{Bernt Ivar Utst$\o$l N$\o$dland}

\begin{document} 

\maketitle
\begin{abstract}
We show that the moduli space of rank three toric vector bundles on smooth projective toric varieties satisfies Murphy's law, that is, every singularity type of finite type arises somewhere on this space.
\end{abstract}

\section{Introduction}
We say that a moduli space satisfies Murphy's law if any singularity type defined over $\Spec \Z$ arises somewhere on the space. This notion was introduced by Vakil \cite{Vakil} who proceeded to show that Murphy's law holds for many moduli spaces. In \cite{PayneModuli} Payne constructed the moduli space of toric vector bundles of fixed rank and equivariant Chern class. He then proceeded to prove that the moduli space of rank three toric vector bundles on quasi-affine toric varieties satisfies Murphy's law. In other words any singularity type arises on some moduli space of rank three toric vector bundles on some quasi-affine toric variety.

In this paper we use similar techniques to prove that the moduli space of rank three toric vector bundles on smooth projective toric varieties also satisfies Murphy's law, thus answering a question of Payne. In contrast, the moduli space of rank two toric vector bundles is a smooth quasi-projective variety \cite{PayneModuli}.

\section{Preliminaries on toric varieties} \label{section:prelimToric}

All of the following preliminary material can be found in any introductory text on toric geometry, for instance \cite{CLS} and \cite{Fulton}. Let $k$ be a field and let $T= (k^\ast)^n$ be an algebraic torus and denote by $M$ its character lattice $\Hom(T,k^\ast)$ and by $N$ its dual lattice of one-parameter subgroups. Both $M$ and $N$ are as groups isomorphic to $\Z^n$ and there is a pairing $M \times N \to \Z$ which, after fixing an isomorphism $M \simeq \Z^n$, is simply the ordinary scalar product. A toric variety  $X$ will in this paper denote a normal irreducible variety containing $T$ as an open dense subset, such that the action of $T$ on itself extends to an action on $X$. It is well-known that toric varieties are classified via combinatorial data.

By a cone $\sigma$ we will mean a strongly convex, rational polyhedral cone $\sigma \subset N_\Q = N \otimes \Q$. There is a bijection between cones $\sigma$, up to $\GL(n,\Z)$, and affine toric varieties obtained by taking the spectrum $U_\sigma$ of the semigroup algebra $ k[\sigma^\vee \cap M]$, where we by $\sigma^\vee$ mean the dual cone of $\sigma$, inside $M \otimes \Q$.

A general toric variety is glued from affine pieces corresponding to cones, which is made precise by the following definition.

\begin{definition}
By a fan we will mean a collection $\Sigma$ of finitely many cones $\sigma$, which is closed under intersections and taking faces.
\end{definition}

There is a bijection between fans, again up to $\GL(n,\Z)$, and toric varieties $X_\Sigma$ obtained by taking the disjoint union of $U_\sigma$ and gluing $U_\sigma$ to $U_\tau$ along the intersection $U_{\sigma \cap \tau}$.

Fundamental to the theory of toric varieties is the following result, describing  orbits (and orbit closures) of the $T$-action on $X_\Sigma$.

\begin{theorem}
There is a bijection between cones $\sigma \subset \Sigma$ and $T$-orbit (closures) in $X_\Sigma$. The correspondence sends a cone of dimension $k$ to a $T$-orbit (closure) of dimension $n-k$. 
\end{theorem}

In particular the cones of dimension $n$ corresponds to $T$-fixed points, while the cones of dimension $1$, called rays, corresponds to $T$-invariant Weil divisors. We often by abuse of notation identify a ray $\rho$ with the unique minimal lattice point on it. We denote by $\Sigma(l)$ the set of $l$-dimensional cones of $\Sigma$ and $\sigma(l)$ the set of $l$-dimensional faces of the cone $\sigma$.

There is an exact sequence describing the divisor class group of a toric variety,
\[ M \to  \Z^{\Sigma(1)} \to \Cl(X_\Sigma) \to 0, \]
which is left exact if the linear span of the rays of $\Sigma$ is $n$-dimensional, for instance if $X_\Sigma$ is projective \cite[Theorem 4.1.3]{CLS}. In particular this says that any divisor is linearly equivalent to a $T$-invariant divisor, in other words to a linear combination of divisors $D_\rho$, where $\rho \in \Sigma(1)$.

On a smooth variety the class group and the Picard group are the same, however if the variety is singular they may be different. The $T$-invariant Cartier divisors is the subgroup $\CDiv_T$ of  $ \Z^{\Sigma(1)} $ consisting of divisors $D = \sum a_\rho D_\rho$ such that for any cone $\sigma \in \Sigma$ there exists a character $m_\sigma \in M$ with $\langle m_\sigma,\rho \rangle = a_\rho$ for any $\rho \in \sigma(1)$. This last technical condition is simply saying that $D$ has to be trivial on $U_\sigma$ and that it has to be the divisor of an invariant rational function on $T$, namely the one corresponding to the character $m_\sigma$.

For a $T$-invariant Cartier divisor $D$ we can also associate a piecewise linear support function 
\[ \phi_D: N \otimes \Q \to \Q \]
which maps any $x \in N \otimes \Q$ to $\langle m_\sigma,x \rangle \in \Q$ for any $\sigma$ containing $x$. This support function uniquely determines $D$.

\section{Toric vector bundles}  \label{section:prelimTVB}

A toric vector bundle $\E$ is a vector bundle on a toric variety $X_\Sigma$ together with a $T$-action on the total space of the vector bundle, making the bundle projection $\E \to X$ into a $T$-equivariant morphism, such that for any $t \in T, x \in X_\Sigma$ the map $\E_x \to \E_{t \cdot x}$ is linear. 
The study of toric vector bundles goes back to Kaneyama \cite{Kaneyama} and Klyachko \cite{Klyachko}, who both gave classifications of toric vector bundles in terms of combinatorial and linear algebra data. Klyachko applied this to study, among other things, splitting of low rank vector bundles on $\p^n$. We here recollect Klyachko's description:

To a toric vector bundle $\E$ of rank $r$ on $X_\Sigma$, we let $E \simeq k^r$ denote the fiber at the identity of the torus. Klyachko shows that there  for each ray $\rho \in \Sigma(1)$ is an associated filtration $E^\rho (j)$ of $E$, indexed over $j \in \Z$. It has the property that for any ray $E^\rho(j) = 0$ for $j$ sufficiently large and $E^\rho(j) = E$ for $j$ sufficiently small. Additionally it satisfies a compatibility condition:

For any maximal cone $\sigma$ there exists characters $u_1,...,u_r \in M$ and vectors $L_u \in E$ such that for any ray $\rho$ of $\sigma$ we have
\[ E^\rho(j)  = \sum_{j | \lb u,\rho \rb \geq j} L_u.\]

The above decomposition is equivalent to the fact that on the affine $U_\sigma$, $\E$ splits into a direct sum of line bundles $\OO(u_i)$. Klyachko's classification theorem is the following:

\begin{theorem} [{\cite[Thm 0.1.1]{Klyachko}}]
The category of toric vector bundles on $X$ is equivalent to the category of finite dimensional vector spaces $E$, with filtrations indexed by the rays as described above, satisfying the compatibility condition. A morphism $\E \to \F$ corresponds to a linear map $E \to F$, respecting the filtrations.
\end{theorem}

Thus, families of vector bundles give rise to families of filtrations of a fixed vector space.

\section{Moduli of toric vector bundles}

In this section we recall Payne's construction of the moduli space of toric vector bundles of fixed equivariant Chern class.

The equivariant Chow ring $A_\ast^T(X_\Sigma)$ of a toric variety $X_\Sigma$ is isomorphic to the ring of integral piecewise polynomial functions, which are polynomial on each cone $\sigma$ \cite[Theorem 1]{PayneEquivariant}. This generalizes the piecewise linear support function associated to a divisor, described above. 

For any cone $\tau \in \Sigma$ we let $M_\tau= M/\tau^\perp \cap M$. If $\sigma$ is a maximal cone containing $\tau$, let $u_1, \ldots ,u_r$ be the characters such that $\E|_{U_\sigma} = \oplus_{i=1}^r \OO(\ddiv(u_i))$. Then the images of $u_1,\ldots,u_r$ in $M(\tau)$ define characters $u_1^\tau,\ldots,u_r^\tau$ in $M_\tau$ which are independent of which maximal cone containing $\tau$ one chooses. Denote by $u(\tau)$ the multiset of characters on the cone $\tau$. Payne showed that the equivariant Chern class $c_i^T(\E)$ of $\E$ is given on $\tau$ by the polynomial $e_i(u(\tau))$, where $e_i$ is the $i$-th elementary symmetric function \cite[Theorem 3]{PayneEquivariant}.

\begin{definition}
A framed rank $r$ toric vector bundle is a toric vector bundle $\E$ together with an isomorphism $\phi:E \to k^r$, where $E$ is the fiber of $\E$ over the identity of the torus.
\end{definition}
Payne proved the following:
\begin{theorem} [{\cite[Theorem 3.9]{PayneModuli}}]
Given a toric variety $X_\Sigma$ and a rank $r$ total equivariant Chern class $c$ there exists a fine moduli scheme $M_c$ of framed rank r toric vector bundles on $X_\Sigma$ with total equivariant Chern class $c$.
\end{theorem}
We note that the scheme $M_c$ is defined over $\Spec \Z$ \cite[p. 1207]{PayneModuli}. The scheme $M_c$ is constructed as a locally closed subscheme of a product of flag varieties, as follows:
Let $\E$ be a framed rank $r$ toric vector bundle of equivariant Chern class $c$. Fixing the Chern class $c$ is equivalent to fixing compatible multisets $u(\sigma)$ of linear functions on each $\sigma$. Recall we have fixed the isomorphism $\phi: E \simeq k^r$. For any ray $\rho$, consider the dimensions of the subspaces appearing in the filtration $E^\rho(j)$. Let $\Fl(\rho)$ be the flag variety of subspaces of $k^r$ having exactly these dimensions. Let $\Fl_c = \prod_\rho \Fl(\rho)$. We see that $\E$ gives an element of $\Fl_c$ by taking the the subspaces appearing in the filtrations for any ray.

Any bundle  with Chern class $c$ must have the same dimensions of subspaces appearing in the filtrations as $\E$ \cite[p. 1205]{PayneModuli}. Moreover the Chern class $c$ restricts how the subspaces  can meet each other in the following way. If $\sigma= \Cone(\rho_1,\ldots,\rho_s)$  we require the equality
\[  \dim \cap_{i=1}^s E^{\rho_i}(j_i) = \# \{ u \in u(\sigma) | \langle u,\rho_i \rangle \leq j_i \text{ for }i=1,\ldots,s\}. \]
These rank conditions correspond to the vanishing of certain polynomials in the Pl\"ucker coordinates of the partial flag varieties $\Fl_\rho$ as well as the non-vanishing of certain others. Thus the resulting scheme is locally closed in  $\Fl_c$ and is in fact the scheme $M_c$.

\section{Murphy's law on smooth projective toric varieties}

Payne has showed that for moduli of bundles of rank $2$, the moduli space $M_c$ is a smooth quasi-projective variety \cite[p.1209]{PayneModuli}. We will show that this fails remarkably for higher rank vector bundles:  For bundles of higher rank the moduli space satisfies Murphy's law, meaning that any singularity type, defined over $\Z$, appears somewhere on $M_c$. The starting observation is that if the bundle has rank $3$ then all the rank conditions correspond to incidences between points and lines in $\p^2$. If $x_1,\ldots,x_d$ are points  and $l_1,\ldots,l_{d'}$ are lines in $\p^2$, then from a subset
\[I \subset \{1,\ldots,d \} \times \{ 1,\ldots,d' \}\]
we can define a set of incidences of points and lines stating that $x_i \in l_j$ if and only if $(i,j) \in I$. There is an associated incidence scheme 
\[ C_I \subset \prod_{i=1}^d \p^2 \times \prod_{i=1}^{d'} {\p^2}^\vee,\]
parametrizing all such sets of points and lines. Payne uses the scheme $C_I$ to prove Murphy's law for moduli of rank three toric vector bundles on quasi-affine and quasi-projective toric varieties. The proof is via, given a set of incidences $I$ between points and lines in $\p^2$, constructing a quasi-affine toric variety such that $M_c$ is $PGL_3$-equivariantly isomorphic to $C_I$ and applying Mn\"{e}v's universality theorem \cite{Mnev}, which states that any singularity appears on some such incidence scheme.  This is essentially the same idea used by  Vakil to prove the original formulation of Murphy's law for other moduli spaces \cite{Vakil}. The way Payne proves this for toric vector bundles is by putting all the points and all the lines as non-trivial subspaces of a filtration $E^\rho(l)$ and then choosing the fan such that all pairs of rays are maximal cones. By construction we get the incidence scheme $C_I$, which has the desired properties.

The material in this section was inspired by the following question:

\begin{question} [{\cite[Remark 4.4]{PayneModuli}}]
Does the moduli of toric vector bundles on projective toric varieties satisfy Murphy's law?
\end{question}
Using techniques similar to the above, we can now prove Murphy's law for moduli of rank at least three toric vector bundles on smooth projective toric varieties.

\begin{theorem}
Given an incidence $I$ between points and lines in $\p^2$ there exists a smooth projective toric variety and a rank three Chern class $c$ on it, such that $M_c$ is $PGL_3$-equivariantly isomorphic to $C_I$. Thus Murphy's law holds for moduli spaces of rank three framed toric vector bundles on smooth projective toric varieties. Also Murphy's law holds for the coarse moduli scheme of rank three toric vector bundles on smooth projective varieties.
\end{theorem}
\begin{proof}
Set $n=d+d'-1$. Let $X_n$ be the toric variety obtained by blowing up $\p^n$ along the following linear spaces of increasing dimension: first all torus-invariant points, then all strict transforms of torus-invariant lines and so on until we have blown up all invariant subvarieties of codimension at least three. We call the new fan $\Sigma$. Blowing up a toric variety corresponds to inserting a new ray in the relative interior of the cone corresponding to the subvariety we blow up. The original $n+1=d+d'$ rays $\rho_1,\ldots,\rho_{n+1}$ of $\p^n$ are thus still rays of $\Sigma$.  Set 
\[y_i=x_i, i=1,\ldots,d \]
\[y_{d+i} = l_{i}, i=1,\ldots,d'.\]
We now specify the equivariant Chern class. In other words we describe the multiset $u(\sigma)$ for each $\sigma \in \Sigma$. For simplicity we first assume that $\sigma = \Cone(\rho_1,\rho_2,\rho_1+\rho_2+\rho_3,\rho_1+\rho_2+\rho_3+\rho_4,\ldots,\rho_1+\ldots+\rho_n)$. Similar to Payne  \cite[Top of p.1211]{PayneModuli} we now define $u(\sigma)$ according to the four cases.
\[ u(\sigma) = \begin{cases} 
\{ 0,e_1^\ast-e_3^\ast,e_2^\ast-e_3^\ast \} \text{ if }y_1,y_2 \text{ are both points} \\
\{ 0,e_2^\ast-e_3^\ast,e_1^\ast+e_2^\ast-2e_3^\ast \} \text{ if }y_1 \text{ is a point containing the line }y_2 \\
\{ e_1^\ast-e_3^\ast,e_2^\ast-e_3^\ast,e_2^\ast-e_3^\ast \} \text{ if } y_1\text{ is a point not contained in the line } y_2  \\
\{ e_1^\ast-e_3^\ast,e_2^\ast-e_3^\ast,e_1^\ast+e_2^\ast-2e_3^\ast \} \text{ if }y_1,y_2 \text{ are both lines} \\
\end{cases} \]
Every maximal cone $\sigma'$ is of the same form as $\sigma$, up to permutation of the $\rho_i$, thus for other maximal cones the definition of $u(\sigma')$ is done in the analogous way.

The above might seem mysterious, however if we fix a bundle $\E$ with equivariant Chern class as above, the point is that  this forces the filtrations on $\rho_i$ to be of the form
\[ E^i(j) = \begin{cases} 
k^3 &\text{ if }  j \leq 0  \\
y_i &\text{ if }  j = 1 \\
0 &\text{ if } 1 < j
\end{cases} \]
and the filtration on any other ray to be trivial, in other words to jump directly from $0$ to $k^3$ at step $0$ of the filtration. It is straight-forward to check that Klyachko's compatibility condition is satisfied for these filtrations: the characters on $U_\sigma$ is exactly the characters $u(\sigma)$. Thus the above filtrations correspond to a toric vector bundle whose Chern classes are the elementary symmetric functions of $u(\sigma)$. Thus $u(\sigma)$ correspond to a well-defined equivariant Chern class. All pairs $\rho_i,\rho_j$ form a two-dimensional cone $\sigma_{ij}$, thus from the Chern class on this cone we get all incidences from $I$. Moreover because we have blown up so much, no three $\rho_i,\rho_j,\rho_k$ form a cone, thus we do not get any extra incidences. Thus $M_c= C_I$ and we are done.

The statement on the coarse moduli scheme follows from the fact that it is the quotient of $M_c$ by $GL_3$ \cite[Corollary 3.11]{PayneModuli}. By Mn\"{e}v's universality theorem \cite[Section 1.8]{MnevThm} we have that for any affine scheme $Y$ defined over $\Spec \Z$, there exists some incidence scheme $C_I$ on which $PGL_3$ acts freely such that the quotient $C_I/GL_3$ is isomorphic to an open subvariety of $Y \times \A^s$ projecting surjectively to $Y$, for some $s$. Thus such quotients satisfy Murphy's law which implies that Murphy's law is satisfied for the coarse moduli scheme of rank three toric vector bundles. This argument is the same as Payne's argument in \cite[Theorem 4.2]{PayneModuli}.
\end{proof}

\begin{remark}
The toric varieties $X_n$ from the proof appear in the literature: If one blows up $X_n$ also along the strict transforms of all codimension two linear spaces, the resulting variety is the Losev--Manin space $LM_n$. By \cite[Remark 1.5]{CastravetTevelev}  the blowup at a general point of $X_{n+1}$ is a small modification of a $\p^1$-bundle over $\overline{M}_{0,n}$. Using this, Castravet and Tevelev prove that the blowup at a general point of $X_n$ is not a Mori Dream Space, for large $n$.
\end{remark}

{\bf Acknowledgements.}
I am grateful to John Christian Ottem for  helpful conversations on the topic in this paper, as well as comments on earlier versions of this paper. I also wish to thank Milena Hering, Nathan Ilten and Ragni Piene for comments on an earlier version of this paper.

\bibliography{ref}

\begin{thebibliography}{Mn{\"e}88}

\bibitem[CLS11]{CLS}
David~A. Cox, John~B. Little, and Henry~K. Schenck.
\newblock {\em Toric varieties}, volume 124 of {\em Graduate Studies in
  Mathematics}.
\newblock American Mathematical Society, Providence, RI, 2011.

\bibitem[CT15]{CastravetTevelev}
Ana-Maria Castravet and Jenia Tevelev.
\newblock {$\overline{M}_{0,n}$} is not a {M}ori dream space.
\newblock {\em Duke Math. J.}, 164(8):1641--1667, 2015.

\bibitem[Ful93]{Fulton}
W.~Fulton.
\newblock {\em Introduction to toric varieties}, volume 131 of {\em Annals of
  Mathematics Studies}.
\newblock Princeton University Press, Princeton, NJ, 1993.
\newblock The William H. Roever Lectures in Geometry.

\bibitem[Kan75]{Kaneyama}
Tamafumi Kaneyama.
\newblock On equivariant vector bundles on an almost homogeneous variety.
\newblock {\em Nagoya Math. J.}, 57:65--86, 1975.

\bibitem[Kly89]{Klyachko}
A.~A. Klyachko.
\newblock Equivariant bundles over toric varieties.
\newblock {\em Izv. Akad. Nauk SSSR Ser. Mat.}, 53(5):1001--1039, 1135, 1989.

\bibitem[Laf03]{MnevThm}
L.~Lafforgue.
\newblock {\em Chirurgie des grassmanniennes}, volume~19 of {\em CRM Monograph
  Series}.
\newblock American Mathematical Society, Providence, RI, 2003.

\bibitem[Mn{\"e}88]{Mnev}
N.~E. Mn{\"e}v.
\newblock The universality theorems on the classification problem of
  configuration varieties and convex polytopes varieties.
\newblock In {\em Topology and geometry---{R}ohlin {S}eminar}, volume 1346 of
  {\em Lecture Notes in Math.}, pages 527--543. Springer, Berlin, 1988.

\bibitem[Pay06]{PayneEquivariant}
Sam Payne.
\newblock Equivariant {C}how cohomology of toric varieties.
\newblock {\em Math. Res. Lett.}, 13(1):29--41, 2006.

\bibitem[Pay08]{PayneModuli}
Sam Payne.
\newblock Moduli of toric vector bundles.
\newblock {\em Compos. Math.}, 144(5):1199--1213, 2008.

\bibitem[Vak06]{Vakil}
Ravi Vakil.
\newblock Murphy's law in algebraic geometry: badly-behaved deformation spaces.
\newblock {\em Invent. Math.}, 164(3):569--590, 2006.

\end{thebibliography}

\bibliographystyle{alpha}
\end{document}